\documentclass[preprint]{imsart}

\usepackage{amsmath,amsthm,amssymb,amsfonts,latexsym,hyperref,dsfont}
\usepackage{stmaryrd}
\usepackage{graphicx}
\usepackage{natbib}
\usepackage{geometry}

\usepackage{mathrsfs}

\newcommand\eps{\varepsilon}
\renewcommand\phi{\varphi}

\newcommand\esp[1]{\mathbb{E}\left[#1\right]}

\newtheorem{thm}{Theorem}[section]
\newtheorem{prop}[thm]{Proposition}

\newtheorem{lem}[thm]{Lemma}
\newtheorem{rem}[thm]{Remark}
\newtheorem{defi}[thm]{Definition}
\newtheorem{ex}[thm]{Example}

\newcommand\n{\mathbb{N}}
\renewcommand\r{\mathbb{R}}

\renewcommand\ll{\left}
\newcommand\rr{\right}

\renewcommand\P{\mathcal{P}}

\newcommand\un{\mathds{1}}
\newcommand\uno[1]{\un_{\left\{#1\right\}}}

\begin{document}

\begin{frontmatter}

\title{The first Fundamental Theorem of Calculus for functions defined on Wasserstein space}

\runtitle{Fundamental theorem of calculus for measure-variable functions}

\begin{aug}

\author{\fnms{Xavier} \snm{Erny}\ead[label=e4]{xavier.erny@telecom-sudparis.eu}}

\address{SAMOVAR, T\'el\'ecom SudParis, Institut Polytechnique de Paris, Palaiseau, France}

\runauthor{X. Erny}
\end{aug}

\begin{abstract}
We study two analytical questions concerning functions defined on Wasserstein space. First, we establish an analogue of the first fundamental theorem of calculus in this setting. Precisely, we show that if a function on Wasserstein space is sufficiently regular in the sense of the linear functional derivative, then its integral is differentiable and the derivative coincides with the integrand. An additional assumption, specific to this framework, is required and turns out to be necessary. Second, we establish a general differentiability criterion connecting Dawson’s derivative, viewed as a Gateaux derivative, with the linear functional derivative, viewed as a Frechet derivative. Under suitable regularity assumptions, Gateaux-differentiability can be upgraded to Fr\'echet-differentiability in the infinite-dimensional setting of Wasserstein space. This criterion is used in the proof of the first fundamental theorem of calculus.

\end{abstract}



\end{frontmatter}

\section*{Introduction}

We study a version of the first fundamental theorem of calculus for measure-variable functions. Roughly speaking, if a measure-variable function is sufficiently regular, then one can define an antiderivative of this function by integrating it. An additional assumption, specific to this framework, turns out to be necessary. In all the paper, we consider functions~$F$ defined on Wasserstein space of order one: the space of probability measures with a finite first order moment, endowed with Wasserstein metric of order one. Understanding the analytical structure of functions defined on Wasserstein space is a natural question, in particular when comparing different notions of differentiability in this infinite-dimensional setting.

One of the first references for the analysis of measure-variable functions is \cite{dawson_measure-valued_1993}. The definition of differentiability introduced by Dawson is close to the following statement, adapted to our framework: a function~$F$ is differentiable at some~$m\in\mathcal P_1(\mathbb R)$ with respect to some~$x\in\mathbb R$ if
\begin{equation}\label{eq:dawson}
\frac{1}{\eps}\left[F\left((1-\eps)m+\eps \delta_x\right)-F(m)\right]
\end{equation}
admits a finite limit as~$\eps$ vanishes. Compared to \cite[p.~19]{dawson_measure-valued_1993}, the weight $1-\eps$ has been added to~$m$ in~\eqref{eq:dawson} in order to consider only probability measures. This definition can be viewed as a partial derivative of~$F$ at the measure~$m$ in the direction~$\delta_x$, and thus as a Gateaux-type derivative.

At least two other notions of derivatives for measure-variable functions have been considered. One is the extrinsic notion introduced by Pierre-Louis Lions (see \cite{cardaliaguet_notes_2013}), which relies on the differentiability of a lift of a function defined on the space of random variables. Another is the linear functional derivative introduced in \cite{carmona_probabilistic_2018}, which is defined directly in terms of perturbations of measures. These notions are closely related and have been used in similar contexts, including mean-field games (\cite{cardaliaguet_notes_2013}), optimal transport (\cite{carmona_probabilistic_2018}), martingale problems and Ito's formula (\cite{mischler_new_2015}, \cite{guo_itos_2023}, \cite{cox_controlled_2024}, \cite{erny_generators_2025}). Other analytical properties of measure-variable functions have also been studied, including Taylor formulas (\cite{chassagneux_weak_2022}), measure-variable polynomials (\cite{cuchiero_probability_2019}), and regularity of stochastic flows of measure-valued processes (\cite{crisan_smoothing_2018}). The study of differentiability notions for functions defined on Wasserstein space therefore arises naturally in several contexts, while remaining essentially analytical in nature.

Notice that \cite{gangbo_differentiability_2019} study connections between several notions of differentiability for functions defined on Wasserstein space, including notions based on the Wasserstein gradient. The derivatives considered in the present paper are instead the Dawson derivative and the linear functional derivative. Both describe the variation of a function when an infinitesimal amount of mass is added at a point~$x$. This is different in nature from the Wasserstein gradient, which describes the variation of a function under an infinitesimal transport of the measure along a vector field. More precisely, the two corresponding perturbations are
$$(1-\eps)m+\eps\delta_x~~\textrm{ and }~~(\textrm{Id}+\eps v)_\#m,$$
respectively, where $f_\#m$ stands for the push-forward measure. Thus, the differentiability notions considered in this paper are of a different nature from those based on the Wasserstein gradient.

In this paper, we address these questions by proving two analytical results. First, we establish a version of the first fundamental theorem of calculus for functions defined on Wasserstein space. Second, we prove a general criterion under which Gateaux-differentiability in the sense of Dawson implies Fr\'echet-differentiability in the sense of the linear functional derivative. Although the second result is used in the proof of the first one, it is of independent analytical interest.

In order to state our results formally, we use the notion {\it linear functional differentiability} introduced by \cite{carmona_probabilistic_2018} (see Definition~5.43). There are two ways to characterize a differentiable function related to this notion. The first one corresponds to a Fr\'echet-differentiability: $F$ is differentiable at~$m_0$ if there exists $\delta F(m_0,\bullet)$ (which is the derivative of~$F$ at $m_0$) such that for all~$m$,
\begin{equation}\label{eq:deriv1}
F(m) - F(m_0) = \int_\r \delta F(m_0,x)\,d(m-m_0)(x) + o\left(W(m_0,m)\right),
\end{equation}
with $W$ some Wasserstein metric. 
{Notice that~\eqref{eq:deriv1} can be written as
$$F(m) - F(m_0) = \left\langle \delta F(m_0,\bullet) , m - m_0\right\rangle + o(W(m,m_0)),$$
hence the function~$\delta F$ corresponds to a Fr\'echet-derivative, where the function~$\delta F(m_0,\bullet)$ can be seen as the gradient of~$F$ at~$m_0$.}

The other characterization is the following: $F$ is differentiable if  there exists $\delta F(m_0,\bullet)$ such that for all~$m_0,m$,
\begin{equation}\label{eq:deriv2}
F(m) - F(m_0) = \int_0^1\int_\r \delta F ((1-t)m_0 + t\,m,x)\,d(m-m_0)(x)\, dt.
\end{equation}

Both characterizations above are equivalent for sufficiently regular functions: if for all~$m_0,m$,~\eqref{eq:deriv1} holds true, then for all~$m_0,m$,~\eqref{eq:deriv2} also holds true (see e.g. Proposition~2.19 of \cite{erny_generators_2025}). This statement can be seen as a version of the second fundamental theorem of calculus (i.e. the increasing of a differentiable function is the integral of its derivative). Conversely, if for all~$m_0,m$,~\eqref{eq:deriv2} is true, then for all~$m_0,m$,~\eqref{eq:deriv1} is true assuming existence and enough regularity of the quantity $\partial_x \delta F(m,x)$ (see e.g. Proposition~5.44 of \cite{carmona_probabilistic_2018}, whose proof is given for Wasserstein space of order two, but can be adapted to Wasserstein space of order one). At first sight, this second statement can be compared to the first fundamental theorem of calculus (i.e. the integral of some (sufficiently regular) function is differentiable, the derivative being the integrand), with one important detail: the second characterization of the differentiability of~$F$ above is ``for all~$m_0$ and all~$m$, \eqref{eq:deriv2} holds true'', whereas a consistent version of the first fundamental would rather be ``if there exists~$m_0$ such that for all $m$, \eqref{eq:deriv2} holds true, then $F$ is differentiable with derivative at~$m_0$ being $\delta F(m_0,\bullet)$''. This is our main result (Theorem~\ref{thm:TFA}): in other words, if \eqref{eq:deriv2} holds for every~$m$ and a particular~$m_0$, then it also holds for every~$m$ and any~$m_0$. In the statement of this result, there are two main assumptions: one about the regularity of the integrand at~\eqref{eq:deriv2} which is standard, and an unexpected one which is specific to the framework of measure-variable functions and necessary (see Remark~\ref{rem:iiidecept} for more explanation about it).


Our second result concerns the relationship between Dawson-differentiability and linear functional differentiability. It corresponds, in our framework, to the classical statement that the existence and sufficient regularity of partial derivatives imply differentiability. Roughly speaking, let $\Delta F(m,x)$ be the limit of~\eqref{eq:dawson} as $\eps$ vanishes (assuming it exists and is finite). If $(m,x)\mapsto \partial_x\Delta F(m,x)$ exists and is jointly continuous, then $F$ is differentiable in the sense of~\eqref{eq:deriv1}, with $\delta F(m,x)=\Delta F(m,x)$ for all~$m$ and~$x$ (see Theorem~\ref{prop:criteria} for the precise statement). Although this criterion is used in the proof of our first result, it is a general differentiability result in its own right. Note that this property would be obvious in a finite-dimensional setting, but we work on the space of probability measures. Our proof relies on an approximation scheme of probability measures by compactly supported purely atomic measures introduced in Section~4.1 of~\cite{cox_controlled_2024}. This scheme allows us to reduce the argument to spaces of purely atomic measures with a fixed finite number of atoms, where two measures differ only through the weights of these atoms. The differentiability in~\eqref{eq:deriv1} and~\eqref{eq:deriv2} can then be deduced from Dawson-differentiability~\eqref{eq:dawson}.

{\bf Organization of the paper.} The next section defines notation that are used throughout the paper. In Section~\ref{sec:def}, we introduce some definitions about the regularity of measure-variable functions and state our two main results. Sections~\ref{sec:proofThmcri} and~\ref{sec:proofTFA} are devoted to the proofs of our main results, respectively Theorems~\ref{prop:criteria} and~\ref{thm:TFA}. In Appendix~\ref{ref:counter}, we provide a counter-example for Theorem~\ref{thm:TFA} when the unusual assumption is not satisfied, and in Appendix~\ref{append:syme}, we prove a lemma justifying formally this assumption in general. Finally in Appendix~\ref{append:proofstech}, we provide the proof of a technical lemma.

\section*{Notation}

\begin{itemize}
\item $\r_+$ is the set of non-negative real numbers.
\item $\mathcal{P}_1(\r)$ is the space of probability measures on~$\r$ with a finite first order moment. It is always endowed with Wasserstein metric of order one, denoted by~$W_1$.
\item The metric $W_1$ is defined as follows: for~$m,\mu\in\P_1(\r),$
$$W_1(m,\mu) := \underset{(X,Y)}{\inf}\esp{|X-Y|},$$
where the infinum ranges over all the pairs of random variables~$(X,Y)$ with respective marginal laws~$m,\mu$.
\item Alternatively, $W_1$ can be defined by Kantorovich-Rubinstein's duality: for~$m,\mu\in\P_1(\r),$
$$W_1(m,\mu) = \underset{f\in \textrm{Lip}_1}{\sup} \int_\r f(x)\,d(m-\mu)(x),$$
with $\textrm{Lip}_1$ the set of $1$-Lipschitz continuous functions~$f:\r\to\r$. Notice that it is possible to restrict the set $\textrm{Lip}_1$ by imposing in addition $f(0)=0$ (possibly by considering $f-f(0)$). In this paper, we mostly use this second characterization of~$W_1$.
\item For $K>0,$ let us denote, by a slight notation abuse, $\P([-K,K])$ the subspace of $\P_1(\r)$ composed of the $[-K,K]$-supported probability measures.
\item For $d\in\n^*$, and $x=(x_1,...,x_d)\in\r^d$, we denote
$$||x||_1 := \sum_{k=1}^d |x_k|\textrm{ and }||x||_\infty = \underset{1\leq k\leq d}{\max}|x_k|,$$
without emphasizing the dependency w.r.t. the dimension~$d$.
\item For $d\in\n^*$, $x=(x_1,...,x_d)\in\r^d$ and $1\leq k\leq d$, we denote
$$x\backslash_k := (x_1,...,x_{k-1},x_{k+1},...,x_d)\in\r^{d-1}.$$
\item For a function $G:(m,x)\in\P_1(\r)\times\r\mapsto G(m,x)$, and $x\in\r$, let us denote
$$G_x:m\in\P_1(\r)\longmapsto G_x(m) := G(m,x).$$
\end{itemize}

\section{Definitions and main results}\label{sec:def}

\subsection{Differentiability of measure-variable functions}

Let us introduce formally the notion of differentiability we use in the paper. This corresponds to what is called ``linear functional differentiability'' as introduced in the Definition~5.43 of \cite{carmona_probabilistic_2018}.

\begin{defi}\label{def:deriv2}
A function~$F:\P_1(\r)\rightarrow\r$ is said to be differentiable if there exists a function~$H:\P_1(\r)\times\r\rightarrow\r$ that is measurable and sublinear w.r.t. the real-variable uniformly w.r.t. the measure-variable belonging to any given compact of $\P_1(\r)$, such that, for all~$m,m_0\in \P_1(\r)$,
\begin{equation}\label{derivtildeH}
F(m) - F(m_0) = \int_0^1\int_\r H ((1-t)m_0 + t\,m,x)\,d(m-m_0)(x)\,\, dt,
\end{equation}
The function $H$ is called a version of the derivative of~$F$. The {\it canonical derivative} of~$F$ at~$m_0$ is defined as the only version verifying, for all~$\mu\in \P_1(\r)$,
\begin{equation}\label{eq:cano}
\int_\r H(\mu,x)\,d\mu(x)=0,
\end{equation}
and is denoted in all the paper by, for all~$\mu\in \P_1(\r),x\in\r$,
$$\delta F(\mu,x) = H(\mu,x).$$
\end{defi}

The proof of our main results require some additional regularity for measure-variable functions. Let us introduce the notion of $C^1$ functions (defined on $\P_1(\r)$) and $C^{0,1}$ functions (defined on $\P_1(\r)\times\r$). 
\begin{defi}
A function~$G:(m,x)\in\P_1(\r)\times\r\mapsto G(m,x)$ is $C^{0,1}$ if it is $C^1$ w.r.t.~$x$ for fixed~$m$, continuous w.r.t.~$m$ for fixed~$x$, sublinear w.r.t.~$x$ uniformly w.r.t.~$m$ belonging to any compact, and if the function $(m,x)\in \P_1(\r)\times \r \mapsto \partial_x G(m,x)$ is jointly continuous.

A function~$F:\P_1(\r)\rightarrow\r$  is $C^1$ if it is differentiable, and if $\delta F$ is $C^{0,1}$.
\end{defi}

In the definition of the class $C^{0,1}$ above, it is wiser to enforce the joint continuity on the function $(m,x)\mapsto\partial_x \delta F(m,x)$ rather than on $(m,x)\mapsto\delta F(m,x)$, since the first one is independent of our ``canonical condition''~\eqref{eq:cano}. So this definition can be used easily in other settings even when~\eqref{eq:cano} is not in force.

Note that the joint continuity condition will be used throughout the paper to bound locally the corresponding functions. As an example, the following  lemma will be useful in our proofs.
\begin{lem}\label{lem:UC}
Let $G:\P_1(\r)\times\r\to\r$ be $C^{0,1}$ and $K>0$. Then $G$ is uniformly continuous w.r.t. the measure-variable~$m$ uniformly w.r.t. the real-variable~$x$ on the set~$\P([-K,K])\times[-K,K]$. 
\end{lem}

\begin{proof}
Since $(m,x)\mapsto \partial_x G(m,x)$ is jointly continuous, it is uniformly continuous on the compact set $\P([-K,K])\times[-K,K]$ (for the compactness of $\P([-K,K])$ in Wasserstein topology, see Proposition~5.3 and Lemma~5.7 of \cite{cardaliaguet_notes_2013}), and the same hold for the function $G_0:m\mapsto G(m,0)$.

And, since, for all~$m\in\P_1(\r)$ and~$x\in\r,$
$$G(m,x) = G(m,0) + \int_0^x \partial_y G(m,y)\,dy,$$
we have, for all~$m,\mu\in\P([-K,K])$ and~$x\in[-K,K]$
\begin{align*}
\left|G(m,x) - G(\mu,x)\right|\leq& |G(m,0) - G(\mu,0)| + \int_{-|x|}^{|x|} |\partial_y G(m,y) - \partial_y G(\mu,y)|dy\\
\leq& \omega_{G_0}(W_1(m,\mu)) + 2|x| \omega(W_1(m,\mu)) \leq \omega_{G_0}(W_1(m,\mu)) + 2 K \omega(W_1(m,\mu)) ,
\end{align*}
with $\omega_{G_0}$ the continuity modulus of $G_0$ restricted to $\P([-K,K])$, and $\omega$ the continuity modulus $(m,x)\mapsto \partial_x G(m,x)$ w.r.t. the measure-variable, restricted to $\P([-K,K])\times [-K,K]$.
\end{proof}

In addition, the class $C^{0,1}$ is used to obtain a practical uniform integrability criterion with the following lemma.
\begin{lem}\label{lem:UI}
Let $G:\P_1(\r)\times\r\rightarrow\r$ be $C^{0,1}$ and $\mathcal{K}$ be a compact set of $\P_1(\r)$. Then the set of functions
$$\mathcal{F}:=\{x\in\r\mapsto G(m,x): m\in\mathcal{K}\}$$
is uniformly integrable w.r.t. any $\mu\in\P_1(\r)$.
\end{lem}

\begin{proof}
Since $m\mapsto G(m,0)$ is continuous, it is bounded on~$\mathcal{K}$. So, there exists~$C>0$ such that, for all $f\in\mathcal{F}$ and $x\in\r,$ $|f(x)|\leq C(1+|x|).$

Whence for any~$K>0$ and $f\in\mathcal{F}$,
$$\int_\r |f(x)|\uno{|f(x)|>K}d\mu(x) \leq C\left(\int_\r \uno{C(1+|x|)>K}d\mu(x) +  \int_\r |x|\uno{C(1+|x|)>K}d\mu(x)\right),$$
which vanishes as $K$ goes to infinity, by the dominated convergence theorem.
\end{proof}

An older notion of measure-variable derivative has been defined by Dawson in \cite[p. 19]{dawson_measure-valued_1993}. As explained in the introduction, the derivative $\delta F$ corresponds to a notion of Fr\'echet-derivative, while the derivative introduced by Dawson (denoted by $\Delta F$ in this paper) corresponds to a notion of Gateaux-derivative. Formally, Dawson has considered the limit of $[F(m_0 + \eps \delta_x) - F(m_0)]/\eps$ as $\eps$ vanishes. Since we are interested in applications where the measures are all probability measures, Definition~\ref{def:derivD} below is more natural for us, and does not require to define a ``canonical continuation'' of functions defined on~$\P_1(\r)$.

\begin{defi}\label{def:derivD}
A function $F:\P_1(\r)\rightarrow\r$ is said to be Dawson-differentiable at~$m_0\in \P_1(\r)$ w.r.t.~$x\in\r$ if
$$\frac1{\eps}\left[F((1-\eps)m_0 + \eps \delta_x) - F(m_0)\right]$$
converges as $\eps>0$ vanishes. In that case, let us denote by $\Delta F(m_0,x)$ this limit. 
\end{defi}

Note that, the Dawson-differentiability is weaker than the differentiability in the sense of Definition~\ref{def:deriv2} according to the following lemma.

\begin{lem}\label{lem:lineardawson}
Every $C^1$ function~$F$ is Dawson-differentiable at any~$m\in \P_1(\r)$ w.r.t. all~$x\in\r$, with
$$\Delta F(m,x) = \delta F(m,x).$$
\end{lem}

\begin{proof}
The proof consists in noticing that, by the differentiability of~$F$ at~$m$,
\begin{align*}
\frac1\eps\left(F((1-\eps)m + \eps \delta_x) - F(m)\right) =& \int_0^1\int_\r \delta F((1-t)m + t(1-\eps)m+t\eps m,y)\, d(\delta_x - m)(y)dt\\
=& \int_0^1\int_\r \delta F((1-t\eps)m + t\eps \delta_x,y)d(\delta_x-m)(y)dt.
\end{align*}

Then, by Vitali's convergence theorem (using Lemma~\ref{lem:UI}), it is possible to make $\eps$ goes to zero to prove that $F$ is Dawson-differentiable at~$m$ with, for $x\in\r$,
$$\Delta F(m,x)= \int_0^1\int_\r \delta F(m,y)~d(\delta_x-m)(y)\,dt = \delta F(m,x),$$
which ends the proof.
\end{proof}

In particular, if $F$ is $C^1$, then, for all~$m\in\P_1(\r)$,
\begin{equation}\label{eq:dawson0}
\int_\r \Delta F(m,x)\, dm(x) = \int_\r \delta F(m,x)\, dm(x) = 0.
\end{equation}

Let us recall that our main result corresponds to a version of the ``first fundamental theorem of calculus'' for measure-variable functions. More precisely, we state that if a function $H:\P_1(\r)\times\r\rightarrow\r$ is regular enough, then it can be written as $H = \delta F$ for an explicit $F:\P_1(\r)\to\r$. In the classical statement of the fundamental theorem of calculus (i.e. for real-variable functions), the continuity of the integrand~$H$ would be sufficient. However, in our framework, we need also $H$ to be differentiable with enough regularity w.r.t. its parameters. So we introduce the class~$C^{1,1}$.
\begin{defi}
A function~$G:(m,x)\in\P_1(\r)\times\r\mapsto G(m,x)$ is $C^{1,1}$ if it is $C^1$ w.r.t.~$x$ for fixed~$m$ and $C^1$ w.r.t.~$m$ for fixed~$x$, and if the functions
\begin{align*}
(m,x)\in \P_1(\r)\times \r& \longmapsto \partial_x G(m,x),\\
(m,x,y)\in\P_1(\r)\times\r^2&\longmapsto \partial_x \delta G_x(m,y),\\
(m,x,y)\in\P_1(\r)\times\r^2&\longmapsto \partial_y \delta G_x(m,y)
\end{align*}
are jointly continuous.
%
\end{defi}

Let us state a lemma similar as Lemma~\ref{lem:UC} above, for the class~$C^{1,1}.$
\begin{lem}\label{lem:UC2}
Let $G:\P_1(\r)\times\r\to\r$ be $C^{1,1}$ and $K>0$. Then the function
$$(m,x,y)\longmapsto \delta G_x(m,y)$$
is uniformly continuous w.r.t. the measure-variable~$m$ uniformly w.r.t. the real-variables~$x,y$ on the set~$\P([-K,K])\times[-K,K]^2$.
\end{lem}

\begin{proof}
Following the computation of the proof of Lemma~\ref{lem:UC}, for all~$m,\mu\in\P([-K,K])$, and~$x,y\in[-K,K]$,
$$\left|\delta G_x(m,y) - \delta G_x(\mu,y)\right|\leq |\delta G_x(m,0) - \delta G_x(m,0)| + 2K\omega(W_1(m,\mu)),$$
with $\omega$ the continuity modulus of $(m,x,y)\mapsto \partial_y\delta G_x(m,y)$ restricted to $\P([-K,K])\times[-K,K]^2$.
Then, since $G$ is $C^{1,1}$, the function $(m,x)\mapsto \delta G_x(m,0)$ is $C^{0,1}$. Whence Lemma~\ref{lem:UC} allows to conclude from the inequality above.
\end{proof}

\subsection{Main results}

The main result of the paper is a kind of ``First Fundamental Theorem of Calculus'' for measure-variable functions (i.e. Theorem~\ref{thm:TFA} below). Roughly speaking, if a function~$H:\P_1(\r)\times\r\rightarrow\r$ satisfies appropriate conditions, we can define an explicit antiderivative of~$H$ by integrating over the measure-variable of~$H$. 

\begin{thm}\label{thm:TFA}
Let $H:\P_1(\r)\times\r\rightarrow\r$ be $C^{0,1}$ satisfying for all~$K>0$:
\begin{itemize}
\item[(i)] $H$ is $C^{1,1}$ on $\P([-K,K])\times[-K,K]$,
 \item[(ii)] for all~$m\in \P([-K,K]),$
 $$\int_\r H(m,x)dm(x)=0,$$
 \item[(iii)] for all~$m\in \P([-K,K]),x,y\in [-K,K],$
 $$\delta H_x(m,y) - H(m,x) = \delta H_y(m,x) - H(m,y).$$
 \end{itemize}
 
 Then the function
 \begin{equation}\label{eq:defF}
 F:m\in\P_1(\r)\longmapsto \int_0^1\int_\r H(t~m + (1-t)\delta_0,x)~d(m-\delta_0)(x)~dt
 \end{equation}
is {differentiable on~$\P_1(\r)$} with: for all~$m\in \P_1(\r),x\in\r,$
 $$\delta F(m,x) = H(m,x).$$
\end{thm}

As expected, the regularity of the integrand~$H$ is required for our main result. While the mere continuity of~$H$ could have been expected to be sufficient, our proof requires a bit more regularity: $H$ is of class $C^{0,1}$ on $\P_1(\r)\times\r$, which is standard in the framework of the linear differentiability (for example, our class of functions~$F:\P_1(\r)\to\r$ $C^1$ is included in the class $\mathcal{C}^{1,1}$ of \cite{guo_itos_2023}, and is the class ``fully $\mathcal{C}^1$'' of \cite{chassagneux_weak_2022}), and of class $C^{1,1}$ at any compactly supported measure. This second regularity condition is needed in our proof.

Condition~$(ii)$ is completely artificial: for any function $H:\P_1(\r)\times\r\rightarrow\r$, the function
$$(m,x)\in\P_1(\r)\times\r\longmapsto H(m,x) - \int_\r H(m,x)\,dm(x)$$
satisfies condition~$(ii)$. We prefer to impose~$(ii)$ to stick to Definition~\ref{def:deriv2}. The last condition~$(iii)$ which does not seem natural is actually a necessary condition (provided the other hypotheses). This is a symmetrical property of the second order derivatives, which is formally proved in Lemma~\ref{lem:intervert}:  since this lemma is useless for the proofs of our main results, this explanation is postponed to Appendix~\ref{append:syme}. We provide a counter-example for Theorem~\ref{thm:TFA} when condition~$(iii)$ is not satisfied at Appendix~\ref{ref:counter}. Notice that the form of the equality in~$(iii)$ above depends on the condition~$(ii)$, but even without~$(ii)$ an identity would persist involving the quantity $\int H(m,\bullet)dm$.

\begin{rem}\label{rem:iiidecept}
Condition~$(iii)$ is deceptively simple and make the proof of Theorem~\ref{thm:TFA} substantially harder. If it is known that $H$ is the derivative of some function (i.e. there exists~$\tilde F:\P_1(\r)\to\r$ such that $H = \delta \tilde F$), then, necessarily, $H$ is a version of the derivative of the function~$F$ defined at~\eqref{eq:defF}: indeed,
$$\tilde F(m) - \tilde F(\delta_0) = \int_0^1\int_\r H((1-t)\delta_0+tm,x)d(m-\delta_0)(x)dt =: F(m),$$
hence (by Lemma~\ref{lemelem}) $\delta F = \delta \tilde F = H$.

However, a function~$H:\P_1(\r)\times\r\to\r$ (even as smooth as possible) is not necessary the derivative of a measure-variable function, because of the necessity of condition~$(iii)$ of Theorem~\ref{thm:TFA}. In particular, the sketch of an easy (but wrong) proof of Theorem~\ref{thm:TFA} would consist in approximating smooth enough functions~$H$ with a particular class of ``practical'' functions by Stone-Weierstrass theorem, in order to work only on this class, and then to obtain the general result by a density argument. But, for this to work, one should guarantee that this class of functions is an algebra included in the set of all the derivatives~$\delta F$ (for $F$ smooth enough), and condition~$(iii)$ makes this class hard to define. If we were able to find such a class, it would be possible to substantially simplify the proof of Theorem~\ref{thm:TFA}. So a subtle consequence of Theorem~\ref{thm:TFA} is a criterion for a function~$H:\P_1(\r)\times\r\to\r$ to be the derivative of a measure-variable function. Let us refer to Example~\ref{ex:counter} for an explicit and smooth function~$H:\P_1(\r)\times\r\to\r$ that cannot be a version of the derivative of a measure-variable function.

Notice that \cite{cox_controlled_2024} have bypassed this problem in their Theorems~4.4 and~4.10 in their framework using the approximation scheme that we reuse at Section~\ref{sec:approx}. But their trick cannot be used to prove Theorem~\ref{thm:TFA} since it requires to work on a function~$F$ that is assumed to be differentiable. 
\end{rem}

The proof of Theorem~\ref{thm:TFA} requires the next result stating that the Dawson-differentiability with $C^{0,1}$ Dawson-derivative implies the differentiability of measure-variable functions. It can be compared to the classical result which claims that existence and continuity of the partial derivatives is a sufficient condition for the Fr\'echet-differentiability for functions defined on~$\r^d$ ($d\in\n^*$). It is a partial converse of Lemma~\ref{lem:lineardawson} above.

\begin{thm}\label{prop:criteria}
Let $F:\P_1(\r)\rightarrow\r$ be continuous. Assume that there exists~$H:\P_1(\r)\times\r\to\r$ of class $C^{0,1}$ such that: for any~$K>0$,
\begin{itemize}
\item[(i)] $F$ is Lipschitz continuous on $\mathcal{P}([-K,K])$,
\item[(ii)] $F$ is uniformly Dawson-differentiable on $\P([-K,K])\times[-K,K]$ with derivative~$H$:
$$\underset{x\in [-K,K], m\in \P([-K,K])}{\sup}\left|\frac1\eps\left(F((1-\eps)m + \eps\delta_x) - F(m)\right) - H(m,x)\right|\underset{\eps\rightarrow 0}{\longrightarrow} 0,$$
 \item[(iii)] for all~$m\in \P([-K,K]),$ $\int_\r H(m,x)~dm(x) = 0$.
 \end{itemize}
 
 Then $F$ is {differentiable on $\P_1(\r)$} with: for all~$m\in \P_1(\r),x\in\r,$
 $$\delta F(m,x) = H(m,x).$$
\end{thm}

Once again, in this kind of result, assumptions about the regularity of~$H$ is expected. As explained previously, the condition~$C^{0,1}$ is very standard in this setting. Condition~$(i)$ which is the local Lipschitz continuity of~$F$ is necessary: if $\delta F= H$ is $C^{0,1}$ then the function $\partial_x\delta F(m,x)$ is bounded on the compact sets, and hence, $\delta F(m,x)$ is Lipschitz continuous w.r.t.~$x$ uniformly w.r.t~$m$, restricting $(m,x)\in\P([-K,K])\times[-K,K]$, and, recalling Definition~\ref{def:deriv2}, Kantorovich-Rubinstein duality implies the desired local Lipschitz continuity. The hypothesis~$(ii)$ about the uniform Dawson-differentiability on some compact sets is stronger than what could be expected, but it is required in our proof. And, as proved at~\eqref{eq:dawson0}, condition~$(iii)$ is also necessary. It is not clear whether the property~$(iii)$ can be deduced from the other assumptions.

Note that, proving that a measure-variable function~$F$ is differentiable can be complicated since the definition of $\delta F$ is implicit. Proposition~\ref{prop:criteria} gives a practical criterion to prove this differentiability reducing it to the Dawson-differentiability, whose corresponding derivative~$\Delta F$ is explicit and coincides with~$\delta F$. This is how Theorem~\ref{prop:criteria} is used in the proof of Theorem~\ref{thm:TFA}.

\section{Proof of Theorem~\ref{prop:criteria}}\label{sec:proofThmcri}

\subsection{Atomic measures approximation}\label{sec:approx}

This section is dedicated to prove that atomic measures are dense in $\P_1(\r)$ providing an explicit convergence scheme. The scheme below is a mere particular case of the one defined at Section~4.1 of \cite{cox_controlled_2024}, that we reproduce here for the sake of clarity.

Let $n\in\n^*$, $K\in\n^*$ and $m\in\P_1(\r)$ whose support is included in~$[-K,K]$. Then, for all $k\in\llbracket -n K+1,n K-1\rrbracket$, let us define
$$I_{n,k} = \ll]\frac{k-1}{n},\frac{k+1}{n}\rr[~~;~~I_{n,-n K} = \ll]-\infty,\frac{-n K+1}{n}\rr[~~;~~I_{n,n K}=\ll]\frac{nK-1}{n},+\infty\rr[.$$

Then $I_{n,k}$ ($-n K\leq k\leq n K$) is an open cover of~$\r$, so we can consider a partition of unity~$(\psi_{n,k})_k$ subordinate to $(I_{n,k})_k$ such that:
\begin{itemize}
\item the functions~$\psi_{n,k}$ are $C^\infty$ and non-negative,
\item $\textrm{Supp}~\psi_{n,k}\subseteq I_{n,k}$,
\item and, for all~$x\in\r,$ $$\sum_{k=-n K}^{n K} \psi_{n,k}(x)=1.$$
\end{itemize}

We write
$$m^{[n]} =  \sum_{k=-n K}^{n K} \left(\int_\r \psi_{n,k}(x)~dm(x)\right) \delta_{k/n}.$$

The following result corresponds to Lemma~4.5.(v)  of \cite{cox_controlled_2024} with an explicit convergence speed. We provide a proof for self-completeness at Appendix~\ref{append:proofstech}.
\begin{prop}\label{prop:mncv}
For all $K\in\n^*$ and~$m\in\P([-K,K])$, for every~$n\geq 1$,
$$W_1\ll(m,m^{[n]}\rr)\leq \frac3n.$$
\end{prop}

\subsection{``Differentiability'' of $F$ on a set of atomic measures}

For~$n\in\n^*,K\in\n^*$, let $D_n^K$ be the following convex set
$$D_n^K = \left\{\sum_{k=-n K}^{n K} \lambda_k \delta_{k/n}~:~\forall -nK\leq k\leq nK,\,\lambda_k\geq 0\textrm{ and }\sum_{l=-nK}^{nK} \lambda_l = 1\right\}.$$

In this subsection, we show that, provided that $F$ is ``uniformly Dawson-differentiable'' on~$D_n^K$, $F$ is ``differentiable'' on $D_n^K$. Note that, formally, it can be problematic to use the term differentiability on $D^K_n$ since it is not an open set. But since it is convex, it is still possible to state the following lemma.

\begin{lem}\label{lem:diffDn}
Let $n\in\n^*,K\in\n^*$, and $F: \P([-K,K])\rightarrow\r$ be Lipschitz continuous on $\P([-K,K])$ and assume that $F$ is ``uniformly Dawson-differentiable on~$\P([-K,K])$'':
\begin{equation}\label{eq:unifDeriv}
\underset{m\in D_n^K,x\in [-K,K]}{\sup}\left|\frac1\eps\left(F\left((1-\eps)m + \eps\delta_x\right) - F(m)\right) - \Delta F(m,x)\right|\underset{\eps\rightarrow 0}{\longrightarrow}0.
\end{equation}
In addition, assume that for all~$m\in D_n^K$, $\int_\r \Delta F(m,x)dm(x)=0$ and that $\Delta F$ is $C^{0,1}$ on $D^K_n$.

Then, for all~$m,\mu\in D_n^K,$
$$F(m) - F(\mu) = \int_0^1 \int_\r \Delta F((1-t)\mu + tm,x)~d(m-\mu)(x)~dt.$$
\end{lem}

\begin{proof}
Let us denote, for $h>0,$
\begin{align*}
\rho_F\left(h\right) :=&~ \underset{\eps\leq h}{\sup}~\underset{m\in D_n^K,x\in [-K,K]}{\sup}~~\left|\frac1\eps\left(F\left((1-\eps)m + \eps\delta_x\right) - F(m)\right) - \Delta F(m,x)\right|,\\
\omega_{\partial F}(h) := &  \underset{x\in[-K,K],m,\mu\in D_n^K, W_1(m,\mu)\leq h}{\sup}~~\left|\partial_x\Delta F(m,x) - \partial_x\Delta F(\mu,x)\right|,\\
\omega_{\Delta F}(h) :=&\underset{x\in[-K,K],m,\mu\in D_n^K, W_1(m,\mu)\leq h}{\sup}~~\left|\Delta F(m,x) - \Delta F(\mu,x)\right|.
\end{align*}


{\it Step~1.} We begin by proving that the three functions above vanish as $h$ goes to zero. For $\rho_F$, it is straightforward by assumption~\eqref{eq:unifDeriv}. For $\omega_{\partial F}$, it is sufficient to note that~$(m,x)\mapsto\partial_x\Delta F(m,x)$ is assumed to be jointly continuous (since $\Delta F$ is $C^{0,1}$) and that $D_n^K\times[-K,K]$ is a compact set (since $D^K_n\subseteq\P([-K,K])$).

Now let us prove that $\omega_{\Delta F}(h)$ vanishes as $h$ goes to zero. Indeed, for any~$\eps>0,$ $x\in\r$ and $m,\mu\in D_n^K$,
\begin{align*} 
\left|\Delta F(m,x) - \Delta F(\mu,x)\right|\leq& \left|\frac1\eps\left(F((1-\eps)m + \eps\delta_x) - F(m)\right) - \Delta F(m,x)\right|\\
&+\left|\frac1\eps\left(F((1-\eps)\mu + \eps\delta_x) - F(\mu)\right) - \Delta F(\mu,x)\right|\\
&+\frac1\eps\left(\left|F(m) - F(\mu)\right| + \left|F((1-\eps)m + \eps\delta_x) - F((1-\eps)\mu + \eps\delta_x)\right|\right),
\end{align*}
implying that, for any $h>0,\eps>0,$
$$\omega_{\Delta F}(h) \leq 2 \rho_F(\eps) + \frac2\eps  L_F h,$$
with $L_F$ a Lipschitz constant for~$F$. So, choosing $\eps = \sqrt{h}$ proves that $\omega_{\Delta F}(h)$ vanishes as $h$ goes to zero.

{\it Step~2.} For $m\in D^K_n$, $d\in\n^*$, $x=(x_1,...,x_d)\in[-K,K]^d$ and $\eps=(\eps_1,...,\eps_d)\in \r_+^d$, let
$$m^x_\eps = \left(1 - \sum_{k=1}^d \eps_k\right) m + \sum_{k=1}^d \eps_k \delta_{x_k}.$$

Let us prove by induction on~$d\in\n^*$ that: for all~$x\in[-K,K]^d,\eps\in\r_+^d,$
\begin{align}
\left|F(m^x_\eps) - F(m) - \int_\r \Delta F(m,y)~d\left(m^x_\eps - m\right)(y)\right| \leq& d \sum_{k=1}^d \eps_k \left(\rho_{F}(2K ||\eps||_\infty)+\omega_{\Delta F}(2K ||\eps||_\infty)\right)\nonumber\\
& + L_F\,2K \sum_{k=1}^d \eps_k \sum_{j=k+1}^d \eps_j.\label{eq:induction}
\end{align}

To prove the case $d=1$, notice that: for $x\in [-K,K],\eps\in\r_+,$
$$\int_\r \Delta F(m,y)~d\left(m^x_\eps - m\right)(y) = \eps \Delta F(m,x) - \int_\r \Delta F(m,y)~dm(y) = \eps \Delta F(m,x),$$

whence
\begin{multline*}
\left|F(m^x_\eps) - F(m) - \int_\r \Delta F(m,y)~d\left(m^x_\eps - m\right)(y)\right|\\
= \eps\left|\frac1\eps\left(F\left((1-\eps)m + \eps\delta_x\right) - F(m)\right) - \Delta F(m,x)\right| \\
\leq \eps~ \rho_{F}(W_1(m,m^x_\eps)) \leq \eps~\rho_{F}(\eps W_1(m,\delta_{x})) \leq\eps~\rho_{F}(2K~\eps)).
\end{multline*}

Let us now prove the induction step. Let $d\geq 2$, $x=(x_1,...,x_d)$ and $\eps = (\eps_1,...,\eps_d)$. Let us recall that we denote
$$x\backslash_1 := (x_2,...,x_d)\in\r^{d-1}\textrm{ and }\eps\backslash_1 := (\eps_2,...,\eps_d)\in\r_+^{d-1}.$$
We have
\begin{align*}
F\left(m^x_\eps\right) - F(m)=& \left[F\left(m^x_\eps\right) - F\left(\left(m^{x_1}_{\eps_1}\right)^{x\backslash_1}_{\eps\backslash_1}\right)\right] + \left[F\left(\left(m^{x_1}_{\eps_1}\right)^{x\backslash_1}_{\eps\backslash_1}\right)- F\left(m^{x_1}_{\eps_1}\right)\right] + \left[ F\left(m^{x_1}_{\eps_1}\right) - F(m)\right]\\
&=: A + B + C.
\end{align*}

Thanks to the case $d=1$, we have
\begin{equation}\label{eq:C}
|C - \eps_1 \Delta F(m,x_1)|\leq \eps_1~\rho_{F}(2K~\eps_1).
\end{equation}

To control~$A$, let us remark that
$$\left(m^{x_1}_{\eps_1}\right)^{x\backslash_1}_{\eps\backslash_1} = m^x_\eps + \eps_1 \sum_{k=2}^d \eps_k\left(m-\delta_{x_k}\right).$$

Then,
\begin{equation}\label{eq:A}
|A|\leq L_F W_1\left(m^x_\eps,\left(m^{x_1}_{\eps_1}\right)^{x\backslash_1}_{\eps\backslash_1}\right) \leq L_F\,\eps_1\sum_{k=2}^d \eps_k W_1(m,\delta_{x_k})\leq L_F\,2K \eps_1\sum_{k=2}^d \eps_k.
\end{equation}

By the induction hypothesis,
$$\left|B - \sum_{k=2}^d \eps_k\Delta F(m^{x_1}_{\eps_1},x_k)\right|\leq (d-1)\sum_{k=2} \eps_k~\left(\rho_{F}(2K~||\eps||_\infty)+\omega_{\Delta F}(2K~||\eps||_\infty)\right) + L_F\,2K \sum_{k=2}^d \eps_k \sum_{j=k+1}^d \eps_j,$$
implying
\begin{align}
\left|B - \sum_{k=2}^d \eps_k\Delta F(m,x_k)\right|\leq&  (d-1)\sum_{k=2} \eps_k \left(\rho_{F}(2K~||\eps||_\infty)+\omega_{\Delta F}(2K~||\eps||_\infty)\right) + L_F\,2K \sum_{k=2}^d \eps_k \sum_{j=k+1}^d \eps_j\nonumber\\
& + \sum_{k=2}^d \eps_k~\omega_{\Delta F}(W_1(m,m^{x_1}_{\eps_1})) \nonumber\\
\leq &d\sum_{k=2}^d \eps_k \left(\rho_{F}(2K~||\eps||_\infty)+\omega_{\Delta F}(2K~||\eps||_\infty)\right) + L_F\,2K \sum_{k=2}^d \eps_k \sum_{j=k+1}^d \eps_j\label{eq:Be}.
\end{align}

Then, combining~\eqref{eq:C}, \eqref{eq:A} and~\eqref{eq:Be} proves the induction step. So~\eqref{eq:induction} is proved for any~$d\in\n^*,$ $x\in[-K,K]^d$ and~$\eps\in\r_+^d$.

{\it Step~3.} Now, we deduce from {\it Step~2} that, for all~$d,K\in\n^*$ and $x\in[-K,K]^d$, for any distinct convex combinations
$$m = \sum_{k=1}^d l_k \delta_{x_k}\textrm{ and }\mu = \sum_{k=1}^d \lambda_k \delta_{x_k},$$
we have
\begin{align}
&\left|F(m) - F(\mu) - \int_\r \Delta F(\mu,y)~d(m-\mu)(y)\right|\label{eq:step2}\\
& \leq C_{d} ||l-\lambda||_1\left(L_F\,2K||l-\lambda||_1 + \rho_{F}\left(2K||l-\lambda||_1\right) + \omega_{\Delta F}\left(2K||l-\lambda||_1 + \omega_{\partial F}(2K ||l-\lambda||_1)\right)\right).\nonumber
\end{align}

Let us define
$$S := \sum_{j=1}^d l_j\wedge\lambda_j~~;~~\tilde m := \sum_{k=1}^d \frac{l_k\wedge \lambda_k}{S}\delta_{x_k},$$
and
\begin{equation*}
\eps_k := l_k-\frac{1-||l-\lambda||_1}{S} l_k\wedge\lambda_k~~;~~\eta_k := \lambda_k-\frac{1-||l-\lambda||_1}{S} l_k\wedge\lambda_k,
\end{equation*}
with $a\wedge b := \min(a,b)$ for $a,b\in\r$.

Notice that, since
$$\lambda_k\wedge l_k = l_k - \max((l_k- \lambda_k),0),$$
we have that
$$S = 1 - \sum_{k=1}^d \max((l_k - \lambda_k),0) \geq 1 - ||l-\lambda||_1,$$
implying that, for all~$1\leq k\leq d,$
$$\eps_k =  l_k-\frac{1-||l-\lambda||_1}{S} l_k\wedge\lambda_k\geq l_k - l_k\wedge\lambda_k\geq 0.$$

And, with the same reasoning, for each~$1\leq k\leq d$, $\eta_k\geq 0$.

Let us also remark that
\begin{equation}\label{eq:DD}
||\eps||_1 = 1 - \frac{1-||l-\lambda||_1}{S}S = ||l-\lambda||_1 = ||\eta||_1.
\end{equation}

And, by definition of $\eps,\eta,\tilde m$, we have
$$m = \left(1-\sum_{k=1}^d \eps_k\right)\tilde m + \sum_{k=1}^d \eps_k\delta_{x_k}\textrm{ and }\mu = \left(1-\sum_{k=1}^d \eta_k\right)\tilde m + \sum_{k=1}^d \eta_k\delta_{x_k}.$$

Besides,
\begin{align}
&\left|F(m) - F(\mu) - \int_\r \Delta F(\mu,y)~d(m-\mu)(y)\right|\leq \left|F(m) - F(\tilde m) - \int_\r \Delta F(\tilde m,y)~d(m-\tilde m)(y)\right|\nonumber\\
&~~+ \left|F(\mu) - F(\tilde m) - \int_\r \Delta F(\tilde m,y)~d(\mu-\tilde m)(y)\right|+\left|\int_\r\left(\Delta F(\tilde m,y) - \Delta F(\mu,y)\right)~d(m-\mu)(y)\right|\label{eq:step2bis}
\end{align}

Then, by {\it Step~2},
\begin{align*}
\left|F(m) - F(\tilde m) - \int_\r \Delta F(\tilde m,y)~d(m-\tilde m)(y)\right| &\leq d ||\eps||_1 \left( \rho_{F}(2K||\eps||_1)+ \omega_{\Delta F}(2K||\eps||_1) + L_F\,2K||\eps||_1\right),\\
 \left|F(\mu) - F(\tilde m) - \int_\r \Delta F(\tilde m,y)~d(\mu-\tilde m)(y)\right| &\leq d ||\eta||_1 \left( \rho_{F}(2K||\eta||_1)+ \omega_{\Delta F}(2K||\eta||_1) + L_F\,2K||\eta||_1\right),
\end{align*}
which allows to control the two first terms of the RHS of~\eqref{eq:step2bis}. By the uniform continuity of $(m',y)\in D^K_n\times[-K,K]\mapsto \partial_y \Delta F(m',y)$, the third term is non-greater than
$$\omega_{\partial F} \left(W_1(\tilde m,\mu)\right) W_1(m,\mu) \leq \omega_{\partial F}\left(2K ||\eta||_1\right)\left(K ||l-\lambda||_1\right).$$

Recalling~\eqref{eq:DD},
$$||\eps||_1=||l-\lambda||_1=|\eta||_1,$$
the inequality~\eqref{eq:step2} is proved.

{\it Step~4.} In this last step, let us fix two distinct convex combinations
$$m = \sum_{k=-n K}^{n K} l_k \delta_{k/n}\textrm{ and }\mu = \sum_{k=-n K}^{n K} \lambda_k \delta_{k/n},$$
and prove that
\begin{equation}\label{eq:step3}
F(m) - F(\mu) = \int_0^1\int_\r \Delta F((1-t)\mu + tm,y)~d(m-\mu)(y)~dt.
\end{equation}

Let us define
$$f : t\in[0,1]\longmapsto F((1-t)\mu + t\,m),$$
and prove that $f$ is differentiable with, for $t\in ]0,1[$,
$$f'(t) = \int_\r  \Delta F((1-t)\mu + t\,m,y)~d(m-\mu)(y).$$

Let $t_0\in]0,1[$. For $h\in\r$ small enough, let
$$R_h := \left|f(t_0+h) - f(t_0) - h\int_\r \Delta F((1-t_0)\mu + t_0 m,y)~d(m-\mu)(y)\right|.$$

We write
$$(1-t_0)\mu + t_0 m = \sum_{k=-nK}^{nK} \left((1 - t_0)\lambda_k + t_0 l_k\right)\delta_{k/n} =:  \sum_{k=-nK}^{nK} \tilde \lambda_k\delta_{k/n},$$
and
$$(1-t_0-h)\mu + (t_0+h) m = \sum_{k=-nK}^{nK} \left((1 - t_0-h)\lambda_k + (t_0+h) l_k\right)\delta_{k/n} =:  \sum_{k=-nK}^{nK} \tilde l_k^h\delta_{k/n}.$$

By {\it Step~3},
\begin{align}\label{eq:Rh}
|R_h|\leq & C_{K,n} ||\tilde l^h-\tilde \lambda||_1\\
&\left(L_F\,2K ||\tilde l^h-\tilde \lambda||_1+\omega_{\Delta F}\left(2K ||\tilde l^h-\tilde \lambda||_1\right) +\rho_{F}\left(2K ||\tilde l^h-\tilde \lambda||_1\right)  +\omega_{\partial F}\left(2K ||\tilde l^h-\tilde \lambda||_1\right) \right).\nonumber
\end{align}

Let us note that, for all~$-nK\leq k\leq nK,$
$$\tilde l^h_k = \tilde \lambda_k + h(l_k-\lambda_k).$$

Hence, we deduce that
$$||\tilde l^h - \tilde \lambda||_1 = h||l - \lambda||_1.$$

Using this last equation in~\eqref{eq:Rh} proves that $R_h/h$ vanishes as $h$ goes to zero. So, $f$ is differentiable at any $t_0\in]0,1[$, and
$$f'(t_0) = \int_\r  \Delta F((1-t_0)\mu + t_0m,y)~d(m-\mu)(y).$$

Then, recall that $(m',y)\mapsto \Delta F(m',y)$ has been shown to be uniformly continuous on the relatively compact set $D^K_n\times [-K,K]$ at {\it Step~1}. This implies that the function $f$ is $C^1$, and so, by the (classical) second fundamental theorem of calculus,
$$f(1) - f(0) = \int_0^1 f'(t)~dt,$$
which is exactly~\eqref{eq:step3}. This ends the proof of the lemma.
\end{proof}

\subsection{End of the proof of Theorem~\ref{prop:criteria}}

Let us begin by stating and proving the following lemma. Note that the proof of the lemma uses similar arguments as some used in the proof of Lemma~$B.1$ of \cite{cox_controlled_2024}.

\begin{lem}\label{lem:PhiCont}
Let $G:\P_1(\r)\times\r\to\r$ be $C^{0,1}$, then the function
$$\Phi:(m,\mu)\in\P_1(\r)^2\longmapsto \int_\r G((1-t)\mu+t\,m,x)~d(m-\mu)(x)~dt$$
is jointly continuous.
\end{lem}

\begin{proof}
Let us consider converging sequences in $\mathcal{P}_1(\r)$ indexed on~$\n^*$,
$$m_n\underset{n\rightarrow\infty}{\longrightarrow}m\textrm{ and }\mu_n\underset{n\rightarrow\infty}{\longrightarrow}\mu,$$
and define
$$\mathcal{K} := \{(1-t)\mu +t\,m:t\in[0,1]\}\cup\bigcup_{n\in\n^*}\left\{(1-t)\mu_n +t\,m_n:t\in[0,1]\right\}$$

{\it Step~1.}  Let us show that $\mathcal{K}$ is a compact set. Let $\nu_k$ belongs to $\mathcal{K}$ ($k\in\n$), meaning: for all~$k\in\n$, there exist~$n_k\in\n,t_k\in[0,1]$ such that
$$\nu_k = (1-t_k)\mu_{n_k} + t_k m_{n_k},$$
with the convention $\mu_0 := \mu$ and $m_0:=m$. Since $[0,1]$ is compact, the sequence $(t_k)_k$ has a converging subsequence. For simplicity, let us assume that $t_k$ converges to some~$t\in[0,1]$. Then $(n_k)_k$ being an $\n-$valued sequence, it is either bounded (hence it admits a constant subsequence) or it has a subsequence going to infinity (hence, along this subsequence, $(m_{n_k},\nu_{n_k})$ converges to $(m,\mu)=(m_0,\mu_0)$). In both cases, there is a subsequence of $n_k$ such that, along it, $(m_{n_k},\nu_{n_k})$ converges to some $(m_N,\mu_N)$ with some~$N\in\n$. Then, if we do not write the subsequence for simplicity, we have
\begin{align*}
W_1(\nu_k,(1-t)\mu_N +t\,m_N)\leq& W_1(\nu_k,(1-t_k)\mu_N+t_km_N) \\
&+ W_1((1-t_k)\mu_N+t_km_N,(1-t)\mu_N+t\,m_N)\\
\leq& t_k \left(W_1(\mu_{n_k},\mu_N)+W_1(m_{n_k},m_N)\right) + |t-t_k|\int_\r|x|d(m_N+\mu_N)(x),
\end{align*}
which proves that, as $k$ goes to infinity, $\nu_k$ converges to $(1-t)\mu_N+t\,m_N$ which belongs to~$\mathcal{K}$. Whence $\mathcal{K}$ is compact.

{\it Step~2.} Now we prove that
$$\underset{n}{\sup}\int_\r (1+|x|)\uno{|x|>K}d(m_n+m)(x)\underset{K\to\infty}{\longrightarrow}0.$$

For $K>0$, let $\phi_K:\r\to[0,1]$ be continuous, $[-K,K]$-supported, and such that, for all~$x\in\r$ $\phi_K(x)$ converges to one as $K$ goes to infinity. Then, for any~$n$,
\begin{align}
\int_\r (1+|x|)\uno{|x|>K}d(m_n+m)(x)\leq& \int_\r (1+|x|)(1-\phi_K(x))d(m_n+m)(x)\nonumber\\
\leq& 2\int_\r (1+|x|)(1-\phi_K(x))dm(x)\label{eq:phiK1}\\
&+\left|\int_\r (1+|x|)(1-\phi_K(x))d(m_n-m)(x)\right|.\label{eq:phiK2}
\end{align}

So, for any $\eps>0$, let us fix $K_\eps>0$ such that~\eqref{eq:phiK1} is non-greater than~$\eps$ (which is possible by the dominated convergence theorem). Then, we can rewrite~\eqref{eq:phiK2} as
$$\left|\int_\r (1+|x|)\phi_{K_\eps}(x)d(m_n-m)(x)\right|$$
 and it is possible to fix some~$n_\eps$ (depending on~$K_\eps$) such that the quantity above is also non-greater than~$\eps$ (since the integrand is continuous and bounded) for any $n\geq n_\eps$.  Then, by the dominated convergence theorem, it is possible to consider some $\tilde K_\eps>K_\eps$ such that, for all~$n\leq n_\eps,$
 $$\int_\r (1+|x|)\uno{|x|>\tilde K_\eps}d(m_n+m)(x)\leq \eps.$$
 Combining the previous controls implies the result.

{\it Step~3.} Let us end the proof by showing
$$\Phi(m_n,\mu_n) \underset{n\to\infty}{\longrightarrow}\Phi(m,\mu).$$

Let us define
$$\mathcal{F} := \{x\in\r\to |G(\nu,x)|:\nu\in\mathcal{K}\}.$$
Since $G$ is $C^{0,1}$ and $\mathcal{K}$ is compact, there exists~$C>0$ such that, for all~$f\in\mathcal{F}$ and $x\in\r$, 
$$f(x)\leq C(1+|x|).$$

In particular, for any~$K>0$,
\begin{align}
\left|\Phi(m,\mu) - \Phi(m_n,\mu_n)\right|\leq& \,\underset{\nu\in \mathcal{K}}{\sup}\int_0^1\int_{[-K,K]^c} |G(\nu,x)| d(m+m_n+\mu+\mu_n)(x)dt\label{eq:1111}\\
&+\int_0^1\int_{[-K,K]} \left|G((1-t)\mu+t\,m,x) - G((1-t)\mu_n+t\,m_n,x)\right|d(m+\mu)(x)dt\label{eq:2222}\\
&+\int_0^1\left|\int_{[-K,K]}G((1-t)\mu_n+t\,m_n,x) d(m-m_n+\mu_n-\mu)(x)\right|dt.\label{eq:3333}
\end{align}

Let us fix some~$\eps>0.$ Since the term at~\eqref{eq:1111} is non-greater than
$$C\int_0^1\int_{[-K,K]^c} (1+|x|) d(m+m_n+\mu+\mu_n)(x)dt,$$
it is possible (by {\it Step~2}) to choose~$K_\eps>0$ such that the term at~\eqref{eq:1111} is smaller than~$\eps$. Then, fixing $K=K_\eps$, the term at~\eqref{eq:2222} vanishes as $n$ goes to infinity by Vitali's convergence theorem using Lemma~\ref{lem:UI}, and hence is smaller than $\eps$ for $n$ large enough (depending on~$K_\eps$). By Kantorovich-Rubinstein duality, the last term~\eqref{eq:3333} is non-greater than
$$\underset{(\nu,x)\in\mathcal{K}\times[-K_\eps,K_\eps]}{\sup}\left|\partial_x G(\nu,x)\right|\left(W_1(m,m_n)+W_1(\mu,\mu_n)\right),$$
which vanishes as $n$ goes to infinity, recalling that $G$ is $C^{0,1}$ and so $\partial_x G$ is bounded on the compact sets. This finally proves that $\Phi$ is jointly continuous.
\end{proof}

Let us finally end this section with the
\begin{proof}[Proof of Theorem~\ref{prop:criteria}]
{\it Step~1.} Let $m,\mu\in\P_1(\r)$ be compactly supported, and let us prove that
\begin{align*}
F(m) - F(\mu) =& \int_0^1\int_\r H((1-t)\mu+t\,m,x)~d(m-\mu)(x)~dt.
\end{align*}

Let us use the notation $m^{[n]},\mu^{[n]}$ introduced at Section~\ref{sec:approx}. Thanks to Lemma~\ref{lem:diffDn}, we know that, for all~$n\in\n^*,$
$$F(m^{[n]}) - F(\mu^{[n]}) = \int_\r H ((1-t)\mu^{[n]}+t\,m^{[n]},x)~d(m^{[n]}-\mu^{[n]})(x)~dt.$$
And, by Lemma~\ref{lem:PhiCont} and Proposition~\ref{prop:mncv}, it is possible to let $n$ goes to infinity to prove the result.

{\it Step~2.} Let $m,\mu$ be any measures in~$\P_1(\r)$. For~$K\in\n^*$ (large enough), let us define
$$m^{(K)} := \frac1{m([-K,K])} m_{|[-K,K]}\textrm{ and }\mu^{(K)} := \frac1{\mu([-K,K])} \mu_{|[-K,K]}.$$

Firstly, let us control Wasserstein distance between $m$ and $m^{(K)}.$ For any Lipschitz continuous function~$\phi:\r\rightarrow\r$ with Lipschitz constant non-greater than one satisfying $\phi(0)=0$,
\begin{align*}
\left|\int_\r \phi(x)~d\left(m-m^{(K)}\right)(x)\right| \leq& \int_{[-K,K]} \left|1 - \frac1{m([-K,K])}\right|\cdot|\phi(x)|~dm(x) + \int_{\r\backslash[-K,K]} |\phi(x)|dm(x)\\
\leq& \left|1 - \frac1{m([-K,K])}\right| \int_\r |x|~dm(x) + \int_{\r\backslash[-K,K]} |x|~dm(x).
\end{align*}

Hence
\begin{equation}\label{eq:w1mk}
W_1\left(m,m^{(K)}\right) \underset{K\rightarrow\infty}{\longrightarrow}0\textrm{ and }W_1\left(\mu,\mu^{(K)}\right) \underset{K\rightarrow\infty}{\longrightarrow}0.
\end{equation}

Then, by {\it Step~1}, for all~$K\in\n^*$ (large enough),
$$F(m^{(K)}) - F(\mu^{(K)}) = \int_\r H((1-t)\mu^{(K)}+t\,m^{(K)},x)~d(m^{(K)}-\mu^{(K)})(x)~dt.$$

Finally, by Lemma~\ref{lem:PhiCont} and~\eqref{eq:w1mk},
$$F(m) - F(\mu) = \int_\r H ((1-t)\mu+t\,m,x)~d(m-\mu)(x)~dt,$$
which ends the proof of Theorem~\ref{prop:criteria}.
\end{proof}

\section{Proof of Theorem~\ref{thm:TFA}}\label{sec:proofTFA}

It is sufficient to prove that the function~$F$ defined as
$$F:m\in \P_1(\r)\longmapsto \int_0^1\int_\r H\left((1-t)\delta_0 + t\,m,x\right)d(m-\delta_0)(x)dt$$
satisfies the assumptions of Theorem~\ref{prop:criteria}, where the function~$H$ in the statements of both Theorems~\ref{thm:TFA} and~\ref{prop:criteria} is the same.


In the statement of both Theorems, the function~$H$ is assumed to be $C^{0,1}$ on $\P_1(\r)$ and to satisfy
$$\int_\r H(m,x)dm(x)=0$$
for all compactly supported measure~$m$. 
In addition, since $H$ is $C^{0,1}$, the function $(m,x)\mapsto \partial_x H(m,x)$ is continuous and hence bounded on the compact set $\mathcal{P}([-K,K])\times[-K,K]$ (for any~$K>0$), and so, by Kantorovich-Rubinstein's duality, for all~$m,\mu\in\mathcal{P}([-K,K]),$
$$\left|F(m) - F(\mu)\right|\leq \underset{\nu\in\P([-K,K]),x\in[-K,K]}{\sup}\left|\partial_x H(\nu,x)\right|\, W_1(m,\mu).$$

So the last property left to be proved is~$(ii)$ of Theorem~\ref{prop:criteria}: for all~$K>0$,
\begin{equation}\label{eq:FH}
\underset{(m,x)\in\P([-K,K])\times[-K,K]}{\sup}\left|\frac1\eps\left[F((1-\eps) m+\eps\delta_x) - F(m)\right]-H(m,x)\right|\underset{\eps\rightarrow 0}{\longrightarrow} 0.
\end{equation}
 
 Let us fix $K>0$, $m\in\P([-K,K])$,~$x\in [-K,K]$, and recall the notation
$$m^x_\eps := (1-\eps)m+\eps\delta_x.$$

We have
\begin{align*}
F(m^x_\eps)=& \int_0^1\int_\r H(t m^x_\eps + (1-t)\delta_0,y)\,d\left((1-\eps)m + \eps\delta_x - \delta_0\right)(y)\,dt\\
=& \int_0^1\int_\r H(t m^x_\eps + (1-t)\delta_0,y)\,d\left[(1-\eps)(m-\delta_0) + \eps(\delta_x - \delta_0)\right](y)\,dt,
\end{align*}
hence
\begin{align*}
\frac{F(m^x_\eps) - F(m)}{\eps}=&\frac{1-\eps}{\eps}\int_0^1\int_\r \left[H(t m^x_\eps + (1-t)\delta_0,y)-H(t m + (1-t)\delta_0,y)\right]\, d(m-\delta_0)(y)\,dt\\
&+\int_0^1\bigg[H(t m^x_\eps + (1-t)\delta_0,x) - H(t m^x_\eps + (1-t)\delta_0,0)\bigg.\\
&\hspace*{2.2cm}\left. + H(t m + (1-t)\delta_0,0) - \int_\r H(t m + (1-t)\delta_0,y)\,dm(y)\right]dt\\
&=: A +B.
\end{align*}

Since $H$ is continuous, it is uniformly continuous on the compact $\P([-K,K])\times[-K,K],$ so
\begin{multline*}
\left|H(t\,m^x_\eps+(1-t)\delta_0,x) - H(t\,m+(1-t)\delta_0,x)\right|\\
\leq \omega_H\left(W_1(t\,m^x_\eps+(1-t)\delta_0,t\,m+(1-t)\delta_0)\right) \leq \omega_H\left(\eps W_1(m,\delta_0)\right)\leq \omega_H(\eps\,K),
\end{multline*}
which vanishes as $\eps$ goes to zero, uniformly w.r.t.~$(m,x)\in\P([-K,K])\times[-K,K].$ This implies that, the following convergence holds true uniformly w.r.t. $(m,x)\in\P([-K,K])\times[-K,K]$:
\begin{equation}\label{eq:B}
B\underset{\eps\rightarrow 0}{\longrightarrow}\int_0^1 H(t\,m+(1-t)\delta_0,x)\,dt - \int_0^1\int_\r H(t\,m+(1-t)\delta_0,y)dm(y)dt.
\end{equation}

Now let us control the term~$A$, by noticing that
$$t\,m^x_\eps + (1-t)\delta_0 = \left(t\,m + (1-t)\delta_0\right)^x_\eps + \eps (1-t) (\delta_0 - \delta_x),$$
 and writing
$$A = A_1 + A_2,$$
with
\begin{align*}
A_1 :=& (1-\eps)\int_0^1\int_\r \frac1\eps\big[H(\left(t\,m + (1-t)\delta_0\right)^x_\eps + \eps (1-t) (\delta_0 - \delta_x),y)\\
&\hspace*{5.5cm} - H(\left(t\,m + (1-t)\delta_0\right)^x_\eps),y)\big]\, d(m-\delta_0)(y)\,dt,\\
A_2 :=& (1-\eps)\int_0^1\int_\r \frac1\eps\left[H(\left(t\,m + (1-t)\delta_0\right)^x_\eps,y) - H(t\,m + (1-t)\delta_0,y)\right]\, d(m-\delta_0)(y)\,dt.
\end{align*}

Since $H$ is $C^{1,1}$ on $\mathcal{P}([-K,K])\times[-K,K]$, by Lemma~\ref{lem:UC2}, the function $(\nu,y,z)\mapsto \delta H_y(\nu,z)$ is uniformly continuous on the compact set $\mathcal{P}([-K,K])\times[-K,K]^2$, hence denoting~$\omega_{\delta H}$ the continuity modulus of this function w.r.t. the measure-variable uniformly w.r.t. the real-variable, we have
\begin{align*}
&\left|\frac1\eps\left(H(\mu^x_\eps,y) - H(\mu,y)\right) - \delta H_y(\mu,x)\right|\\
&=\left|\int_0^1\int_\r \delta H_y((1-s\eps)\mu + s\eps\delta_x,z) d(\delta_x -\mu)(z)ds - \delta H_y(\mu,x) + \int_\r \delta H_y(\mu,z)d\mu(z)\right| \\
&\leq \int_0^1\int_\r\left| \delta H_y((1-s\eps)\mu + s\eps\delta_x,z) - \delta H_y(\mu,z)\right|d(m+\delta_x)(z)ds\\
&\leq 2\int_0^1\omega_{\delta H}\left(W_1((1-s\eps)\mu+s\eps\delta_x,\mu)\right)ds\leq 2 \int_0^1\omega_{\delta H}\left(s\eps W_1(\mu,\delta_x)\right)ds\\
&\leq 2\,\omega_{\delta H}(\eps W_1(\mu,\delta_x)) \leq 2\,\omega_{\delta H}(2K\,\eps).
\end{align*}
In particular,
\begin{align}
A_2\underset{\eps\rightarrow 0}{\longrightarrow}& \int_0^1 \int_\r \delta H_y (tm + (1-t)\delta_0,x)\, d(m-\delta_0)(y)\,dt\label{eq:A2}\\
=& \int_0^1 \int_\r \delta H_x (t\,m + (1-t)\delta_0,y)\, d(m-\delta_0)(y)\,dt\nonumber\\
&+\int_0^1 \int_\r H(tm+(1-t)\delta_0,y)dm(y)dt-\int_0^1 H(t\,m+(1-t)\delta_0,0)dt\nonumber,
\end{align}
where the convergence is uniform w.r.t.~$(m,x)$ belonging to $\P([-K,K])\times[-K,K]$.

To handle~$A_1,$ let us remark that
\begin{multline*}
H(\mu + \eps(1-t)(\delta_0-\delta_x),y) - H(\mu,y)\\
=\eps(1-t)\int_0^1\int_\r \delta H_y(\mu + \eps(1-t)(1-s)(\delta_0-\delta_x),z)d(\delta_0-\delta_x)(z)ds,
\end{multline*}
hence
\begin{multline*}
A_1=(1-\eps)\int_0^1(1-t)\int_\r \int_0^1\int_\r\\
 \delta H_y\left((t\,m+(1-t)\delta_0)^x_\eps+ \eps(1-t)(1-s)(\delta_0-\delta_x),z \right)d(\delta_0-\delta_x)(z)ds\, d(m-\delta_0)(y)dt.
\end{multline*}

So with the same reasoning as the one used to prove the convergence~\eqref{eq:B} (applied to the function~$(\tilde m,y,z)\mapsto \delta H_y(\tilde m,z)$ instead of $(\tilde m,y)\mapsto H(\tilde m,y)$, which is uniformly continuous by Lemma~\ref{lem:UC2}) we obtain the following convergence (and this convergence is uniform w.r.t.~$(m,x)$ in $\mathcal{P}([-K,K])\times[-K,K]$)
\begin{align}
A_1\underset{\eps\rightarrow 0}{\longrightarrow}& \int_0^1 (1-t)\int_\r \int_\r \delta H_y(t\,m + (1-t)\delta_0,z)\, d(\delta_0 - \delta_x)(z)\,d(m-\delta_0)(y)\,dt\nonumber\\
=&\int_0^1 (1-t)\int_\r \int_\r \delta H_z(t\,m + (1-t)\delta_0,y)\, d(\delta_0 - \delta_x)(z)\,d(m-\delta_0)(y)\,dt\nonumber\\
=& \int_0^1 (1-t)\int_\r \int_\r \delta H_z(t\,m + (1-t)\delta_0,y)\, d(m-\delta_0)(y)\, d(\delta_0 - \delta_x)(z)\,dt\nonumber\\
=&\int_0^1 \int_\r \int_\r \delta H_z(t\,m + (1-t)\delta_0,y)\, d\left[m - (t\,m + (1-t)\delta_0)\right](y)\, d(\delta_0 - \delta_x)(z)\,dt\nonumber\\
=& \int_0^1 \int_\r\int_\r \delta H_z(t\,m + (1-t)\delta_0,y)\, dm(y)\,d(\delta_0 - \delta_x)(z)\nonumber\\
=&  -\int_0^1 \int_\r \delta H_x(t\,m + (1-t)\delta_0,y)\, dm(y)\,dt +\int_0^1\int_\r \delta H_0(t\,m +(1-t)\delta_0,y)\,dm(y)\,dt,\label{eq:A1}
\end{align}
where we have used~$(iii)$ to obtain the first equality above.

Then, by~\eqref{eq:B},~\eqref{eq:A2} and~\eqref{eq:A1}, the function~$F$ is uniformly Dawson-differentiable at~$m$ on $\mathcal{P}([-K,K])\times[-K,K]$ with
\begin{align}
\Delta F(m,x) = &\int_0^1 H(t\,m+(1-t)\delta_0,x)\,dt - \int_0^1\int_\r H(t\,m+(1-t)\delta_0,y)dm(y)dt \nonumber\\
&+\int_0^1 \int_\r \delta H_x (t\,m + (1-t)\delta_0,y)\, d(m-\delta_0)(y)\,dt\nonumber\\
&+\int_0^1 \int_\r H(t\,m+(1-t)\delta_0,y)dm(y)dt-\int_0^1 H(t\,m+(1-t)\delta_0,0)dt\nonumber\\
& -\int_0^1 \int_\r \delta H_x(t\,m + (1-t)\delta_0,y)\, dm(y)\,dt +\int_0^1\int_\r \delta H_0(t\,m +(1-t)\delta_0,y)\,dm(y)\,dt\nonumber\\
=&  \int_0^1 H_x(t\,m + (1-t)\delta_0)\,dt - \int_0^1 \delta H_x(t\,m + (1-t)\delta_0,0)\,dt\label{eq:fin1}\\
&-\int_0^1 H_0(t\,m+(1-t)\delta_0)dt + \int_0^1\int_\r \delta H_0(t\,m+(1-t)\delta_0,y)dm(y)dt.\label{eq:fin2}
\end{align}

To end the proof, let us show that the quantity at the line~\eqref{eq:fin1} is $H(m,x)$ and the one at~$\eqref{eq:fin2}$ is zero. We have
\begin{align*}
H_x(t\,m + (1-t)\delta_0) - H_x(m) =& \int_0^1\int_\r \delta H_x((1-s)m + st\, m + s(1-t)\, \delta_0,y) (1-t)\, d(\delta_0-m)(y)\, ds\\
=&\int_0^1\int_\r \delta H_x((1-s(1-t))m + s(1-t)\, \delta_0,y) (1-t)\, d(\delta_0-m)(y)\, ds\\
=&\int_0^{1-t}\int_\r \delta H_x((1-r)m + r\,\delta_0,y)\, d(\delta_0-m)(y)\, dr\\
=&\int_t^1 \int_\r \delta H_x(s\, m + (1-s)\, \delta_0,y)d(\delta_0-m)(y)\, ds.
\end{align*}

Since, by Kantorovich-Rubinstein's duality,
\begin{multline*}
\int_0^1\int_0^1\left|\delta H_x(s\,m +(1-s)\delta_0,y)d(\delta_0-m)(y)\right|ds\,dt \\
\leq \left(\underset{(\mu,y,z)\in\P([-K,K])\times[-K,K]^2}{\sup}\left|\partial_y \delta H_z(\mu,y)\right|\right)W_1(m,\delta_0),
\end{multline*}
we can apply Fubini-Lebesgue's theorem to write
\begin{align*}
\int_0^1\left[H_x(t\,m + (1-t)\delta_0) - H_x(m)\right]dt =& \int_0^1\int_0^1\uno{s\geq t}\int_\r \delta H_x(s\, m + (1-s)\, \delta_0,y)d(\delta_0-m)(y)\, ds\,dt\\
=& \int_0^1 \int_0^s \int_\r \delta H_x(s\, m + (1-s)\, \delta_0,y)d(\delta_0-m)(y)\,dt\,ds\\
=&\int_0^1 s\int_\r \delta H_x(s\, m + (1-s)\, \delta_0,y)d(\delta_0-m)(y)\, ds\\
=&\int_0^1 \int_\r \delta H_x(s\, m + (1-s)\, \delta_0,y)d\left[s(\delta_0-m)\right](y)\, ds\\
=&\int_0^1 \int_\r \delta H_x(s\, m + (1-s)\, \delta_0,y)d\left[\delta_0 - (s\,m+(1-s)\delta_0)\right](y)\, ds\\
=& \int_0^1 \delta H_x(s\,m + (1-s)\delta_0,0)ds.
\end{align*}

As a consequence
\begin{equation}\label{eq:finfin1}
\int_0^1 H_x(t\,m + (1-t)\delta_0)\,dt - \int_0^1 \delta H_x(t\,m + (1-t)\delta_0,0)\,dt = H_x(m).
\end{equation}
In particular, for $x=0$ we get
$$\int_0^1 H_0(t\,m + (1-t)\delta_0)\,dt - \int_0^1 \delta H_0(t\,m + (1-t)\delta_0,0)\,dt = H_0(m).$$

Using this last equation, we can write~\eqref{eq:fin2} as
\begin{align}
&\int_0^1\int_\r \delta H_0(t\,m+(1-t)\delta_0,y)dm(y)dt-\int_0^1 H_0(tm+(1-t)\delta_0)dt\nonumber\\
&= \int_0^1\int_\r \delta H_0(t\,m+(1-t)\delta_0,y)d(m-\delta_0)(y)dt + H_0(m)\nonumber\\
&= H_0(m) - H_0(\delta_0) - H_0(m) = 0.\label{eq:finfin2}
\end{align}

Then, using respectively~\eqref{eq:finfin1} and~\eqref{eq:finfin2} in~\eqref{eq:fin1} and~\eqref{eq:fin2}, we obtain that: for all~$m\in\P([-K,K])$ and~$x\in[-K,K]$,
$$\Delta F(m,x) = H(m,x).$$

Finally, Theorem~\ref{prop:criteria} allows to conclude the proof of Theorem~\ref{thm:TFA}.

\begin{appendix}

\section{Counter-example for Theorem~\ref{thm:TFA}}\label{ref:counter}

In this section, we provide a counter-example for Theorem~\ref{thm:TFA} when condition~$(iii)$ is not satisfied. For the sake of simplicity, we work under the following definition: a function~$F$ is differentiable at~$m_0$ if for all~$m$ in a neighborhood of $m_0$,
\begin{equation}\label{eq:vraidef}
F(m) = F(m_0) + \int_\r \delta F(m_0,x)\,d(m-m_0)(x) + o\left(W_1(m,m_0)\right).
\end{equation}

Notice that, if a function~$F$ is differentiable on some convex set~$D$ in the sense~\eqref{eq:vraidef}, and if $F$ belongs to $C^1$ on~$D$, then $F$ is differentiable in the sense of Definition~\ref{def:deriv2} (i.e. the notation of differentiability that we use in all the paper). This is a straightforward consequence of the fact that, under these assumptions, for any~$m,m_0\in D$, the function
$$f : t\in[0,1]\longmapsto F\left((1-t)m_0 + t\,m\right)$$
is $C^1$ with
$$f':t\in[0,1]\longmapsto \int_\r \delta F\left((1-t)m_0 + t\,m,x\right)\,d(m-m_0)(x),$$
hence
$$F(m) - F(m_0) = f(1) - f(0) = \int_0^1 f'(t)dt = \int_0^1 \int_\r \delta F\left((1-t)m_0 + t\,m,x\right)\,d(m-m_0)(x).$$

Now let us give an elementary example of a differentiable function.
\begin{ex}\label{ex:monome}
For any $\phi:C^1_b(\r)$, the function
$$F_\phi:m\in\P_1(\r)\longmapsto \int_\r \phi(x)\,dm(x)$$
is differentiable (in the sense~\eqref{eq:vraidef}) and is $C^1$ with: for all~$m\in\P_1(\r),x\in\r,$
$$\delta F_\phi(m,x) = \phi(x) - \int_\r \phi(y)\,dm(y),$$
since, for any~$m,\mu\in\P_1(\r),$
$$F_\phi(m) - F_\phi(\mu) - \int_\r \phi(x)\,d(m-\mu)(x)dt = 0 .$$
\end{ex}

In order to obtain the counter-example, the following lemma is also required. It states that the differential operator~$\delta$ satisfies usual conditions. Note that this lemma is only true under the ``canonical condition''~\eqref{eq:cano}. We refer to Lemma~2.4 of \cite{erny_generators_2025} for the proof.
\begin{lem}[Lemma~2.4 of \cite{erny_generators_2025}]\label{lemelem} Let $D\subseteq\mathcal{P}_1(\r)$ be an open convex set.
\begin{itemize}
\item[$(a)$] Let $F: D\rightarrow\r$ be a constant function. Then $F$ is differentiable (in the sense~\eqref{eq:vraidef}) at any~$m\in D$, and its derivative is the zero function.
\item[$(b)$] Let $F,G:D\rightarrow\r$ be differentiable (in the sense~\eqref{eq:vraidef}) at~$m_0\in D,$ and~$\alpha\in\r$. Then, the function~$\alpha F+G$ is differentiable (in the sense~\eqref{eq:vraidef}) at~$m_0$, and, for all~$x\in\r$,
$$\delta (\alpha F + G) (m_0,x) = \alpha\delta F(m_0,x) + \delta G(m_0,x).$$
\item[$(c)$] Let $F,G:D\rightarrow\r$ be Lipschitz continuous and differentiable on~$D$ (in the sense~\eqref{eq:vraidef}). Then, the product function $FG$ is differentiable (in the sense~\eqref{eq:vraidef}) at any~$m_0\in D$, with, for all~$x\in\r$,
$$\delta (FG)(m_0,x) = F(m_0)\delta G(m_0,x) + G(m_0)\delta F(m_0,x).$$
\end{itemize}
\end{lem}

Let us end this section with the following counter-example of Theorem~\ref{thm:TFA} when condition~$(iii)$ is not verified.
\begin{ex}\label{ex:counter}
Let $\phi,\psi$ belong to $C^1_b(\r)$ 
and define
\begin{align*}
H:&(m,x)\in\P_1(\r)\times\r\longmapsto \left[\phi(x) - \int_\r \phi(y)\,dm(y)\right]\left(\int_\r \psi(y)\,dm(y)\right),\\
F:&m\in\P_1(\r)\longmapsto \int_0^1 \int_\r H((1-t)\delta_0 + t\,m,x)\,d(m-\delta_0)(x)\,dt.
\end{align*}

By the results proved in this section, $H$ is differentiable with: for any~$m\in\P_1(\r),x,y\in\r,$
$$\delta H_x(m,y) = \left[\phi(x) - \int_\r \phi\, dm\right] \left[\psi(y) - \int_\r \psi\,dm\right] - \left[\phi(y) - \int_\r \phi\,dm\right] \int_\r \psi\,dm,$$
hence all the conditions of Theorem~\ref{thm:TFA} (except $(iii)$) are trivially satisfied.

For all~$m\in\P_1(\r)$,
\begin{align*}
F(m)=& \int_0^1 \left(\int_\r \psi(y)~d((1-t)\delta_0+tm)(y)\right)\int_\r\phi(x)\, d(m-\delta_0)(x)\,dt\\
=&\int_0^1\left((1-t)\psi(0) + t\int_\r \psi(y)\,dm(y)\right)dt\left(\int_\r\phi(x)\,dm(x) - \phi(0)\right)\\
=& \frac12\left(\psi(0) + \int_\r \psi(y)\,dm(y)\right)\left(\int_\r\phi(x)\,dm(x) - \phi(0)\right),
\end{align*}
hence, for all~$m\in\P_1(\r),x\in\r$,
\begin{align*}
\delta F(m,x) =& \frac12\left(\psi(x) - \int_\r \psi(y)\,dm(y)\right)\left(\int_\r\phi(y)\,dm(y) - \phi(0)\right)\\
&+\frac12\left(\psi(0) + \int_\r \psi(y)\,dm(y)\right)\left(\phi(x) - \int_\r \phi(y)\,dm(y)\right).
\end{align*}

For simplicity, let us write
\begin{align*}
H(m,x) =& \phi(x)\int_\r \psi\,dm + C_1(m)\\
\delta F(m,x) =& \frac12 \phi(x)\left[\int_\r \psi\,dm + \psi(0)\right] + \frac12\psi(x)\left[\int_\r\phi\,dm - \phi(0)\right] +C_2(m),
\end{align*}

It is clear that both functions~$H$ and~$\delta F$ are different in general, and the same statement holds for $\partial_x H$ and~$\partial_x \delta F$. 

Roughly speaking, 
$\delta F$ seems to be a ``symmetrized version'' of $H$ w.r.t.~$(\phi,\psi)$. This symmetrical property corresponds to condition~$(iii)$ of Theorem~\ref{thm:TFA}, which comes from Lemma~\ref{lem:intervert} below. As explained in Remark~\ref{rem:iiidecept}, this implies that there exists no function~$\tilde F$ such that $H=\tilde F$, otherwise, $H$ would necessarily be a version of the derivative of the function~$F$ defined above.
\end{ex}

\section{Symmetrical property of the second-order derivative}\label{append:syme}

The aim of this section is to justify assumption~$(iii)$ of Theorem~\ref{thm:TFA}. Formally, let us recall that the aim of Theorem~\ref{thm:TFA} is to prove that, under some hypotheses, a function $H:\P_1(\r)\times\r\rightarrow\r$ is the derivative of some function~$F:\P_1(\r)\times\r$ (i.e. for all $m\in\P_1(\r),x\in\r,$ $H(m,x) = \delta F(m,x)$). Applying Lemma~\ref{lem:intervert} below to the function~$F$ proves that condition~$(iii)$ of Theorem~\ref{thm:TFA} is necessary, provided that $H$ is $C^{1,1}$.

To state and prove this result, a suitable definition of ``twice differentiability'' and some regularity for measure-variable functions are required.

\begin{defi}
A function $F:\P_1(\r)\rightarrow\r$ is said to be twice differentiable if it is differentiable, and if, for any~$x\in\r,$ $m\in \P_1(\r)\mapsto \delta F(m,x)$ is also differentiable. In this case, let us denote, for $m\in \P_1(\r),x,y\in\r,$
$$\delta^2 F(m,x,y) = \delta (\delta F)_x(m,y).$$

A function~$F:\P_1(\r)\rightarrow\r$ is $C^2$ if it is twice differentiable such that, the functions
$$m\in \P_1(\r)\mapsto F(m);~(m,x)\in \P_1(\r)\times\r\mapsto \delta F(m,x);~(m,x,y)\in \P_1(\r)\times\r^2\mapsto \delta^2 F(m,x,y)$$
are continuous w.r.t.~$m$ and $C^1$ w.r.t. the real variables, and if the functions
\begin{align*}
(m,x)\in\P_1(\r)\times\r&\longmapsto\partial_x\delta F(m,x),\\
(m,x,y)\in\P_1(\r)\times\r^2&\longmapsto\partial_x \delta^2F(m,x,y),\\
(m,x,y)\in\P_1(\r)\times\r^2&\longmapsto\partial_y \delta^2F(m,x,y)
\end{align*}
are jointly continuous.
\end{defi}

\begin{lem}\label{lem:intervert}
Let $F:\P_1(\r)\to\r$ be $C^2$. Then, for all~$x,y\in\r$ and $m\in \P_1(\r)$,
$$\delta^2 F(m,x,y) - \delta F(m,x) = \delta^2 F(m,y,x) - \delta F(m,y).$$
\end{lem}

Let us begin by stating the following lemma which is almost the same as Lemma~2.15 of \cite{erny_generators_2025}. The only difference is the definition of the linear differentiability (Definition~\ref{def:deriv2} here rather than~\eqref{eq:deriv1}). The proof being almost the same, it is omitted

\begin{lem}\label{lem:215}
Let $F$ be $C^1$, and $m_0,m_1\in \P_1(\r)$. Then, the function
$$f : t\in[0,1]\longmapsto F\left((1-t)m_0 + t m_1\right)$$
is differentiable on~$[0,1]$ with, for all~$t\in[0,1],$
$$f'(t)=\int_\r \delta F\left((1-t)m_0 + t m_1,x\right)\,d(m_1 - m_0)(x).$$
\end{lem}

Now, we turn to the
\begin{proof}[Proof of Lemma~\ref{lem:intervert}]
Let us fix some~$m\in\P_1(\r)$,~$x,y\in\r$. Using twice Lemma~\ref{lem:lineardawson},
\begin{align*}
\delta^2 F(m,x,y)=&\delta ((\delta F)_x)(m,y)\\
 =&\ll(\partial_{\eta_1} (\delta F)_x ((1-\eta_1) m + \eta_1 \delta_y)\rr)_{|\eta_1=0} \\
=& \partial^2_{\eta_1\eta_2} F((1-\eta_2)[(1-\eta_1)m + \eta_1\delta_y] + \eta_2\delta_x)_{|\eta_1=\eta_2=0}\\
=& \partial^2_{\eta_1\eta_2} F((1-\eta_2)(1-\eta_1)m + \eta_1\delta_y + \eta_2\delta_x -\eta_1\eta_2\delta_y)_{|\eta_1=\eta_2=0}.
\end{align*}

Using Lemma~\ref{lem:215} multiple times, it can be shown, thanks to the assumption that $F$ is $C^2$, that
$$(\eta_1,\eta_2)\longmapsto F((1-\eta_2)(1-\eta_1)m + \eta_1\delta_y + \eta_2\delta_x -\eta_1\eta_2\delta_y)$$
belongs to~$C^2(\r^2)$. Hence, by Schwarz' theorem, it is possible to exchange the role of~$\eta_1$ and~$\eta_2$:
$$\delta^2 F(m,x,y) = \partial^2_{\eta_1\eta_2} G((1-\eta_2)(1-\eta_1)m + \eta_1\delta_x + \eta_2\delta_y -\eta_1\eta_2\delta_y)_{|\eta_1=\eta_2=0}.$$

Let us denote
$$\mu(\eta_1,\eta_2) = (1-\eta_2)(1-\eta_1)m + \eta_1\delta_x + \eta_2\delta_y-\eta_1\eta_2\delta_y\textrm{ and }\nu(\eta_1) = \mu(\eta_1,0) = (1-\eta_1)m + \eta_1\delta_x.$$

By noticing that
$$\mu(\eta_1,\eta_2) = (1-\eta_2)\nu(\eta_1) + \eta_2\ll(\eta_1 \delta_x + (1-\eta_1)\delta_y\rr),$$
we have, by Lemma~\ref{lem:215} (with $m_0 = \nu(\eta_1)$ and~$m_1 = \eta_1 \delta_x + (1-\eta_1)\delta_y$),
\begin{align*}
\partial_{\eta_2} F(\mu(\eta_1,\eta_2))_{|\eta_2 = 0} =& \int_\r \delta F(\nu(\eta_1),z)d((1-\eta_1)(\delta_y - m))(z) \\
=& (1-\eta_1)\delta F(\nu(\eta_1),y) - (1-\eta_1)\int_\r \delta F(\nu(\eta_1),z)dm(z).
\end{align*}

So,
\begin{align*}
\partial_{\eta_1}\ll(\partial_{\eta_2} F(\mu(\eta_1,\eta_2))_{|\eta_2 = 0}\rr) =& -\delta F(\nu(\eta_1),y) + (1-\eta_1)\partial_{\eta_1}\delta F(\nu(\eta_1),y)\\
&+\int_\r \delta F(\nu(\eta_1),z)dm(z) -(1-\eta_1)\int_\r  \partial_{\eta_1} \delta F(\nu(\eta_1),z)dm(z),
\end{align*}
where we have differentiated under the integral sign for the last term above. This is legit since, using once again Lemma~\ref{lem:215} (with $m_0=m$ and $m_1=\delta_x$) for all~$z\in\r$,
$$\partial_{\eta_1} \delta F(\nu(\eta_1,z) = \partial_{\eta_1} \delta F(\nu((1-\eta_1)m + \eta_1\delta_x,z) = \int_\r \delta ((\delta F)_z)(\nu(\eta_1),z_0) d(\delta_x - m)(z_0).$$
Consequently, recalling that $\delta F$ is the canonical derivative of~$F$, for all~$z\in\r$,
$$\partial_{\eta_1} \delta F(\nu(\eta_1,z))_{|\eta_1=0} = \delta ((\delta F)_z)(m,x) - \int_\r \delta((\delta F)_z)(m,z_0)dm(z_0) = \delta ((\delta F)_z)(m,x).$$

Whence
\begin{align*}
\partial^2_{\eta_1\eta_2} F(\mu(\eta_1,\eta_2))_{|\eta_1=\eta_2 = 0} =& -\delta F(m,y) + \delta ((\delta F)_y)(m,x)\\
&+\int_\r \delta F(m,z)dm(z) -\int_\r  \delta ((\delta F)_z)(m,x)dm(z)\\
=& -\delta F(m,y) + \delta^2 F(m,y,x)-\int_\r  \delta ((\delta F)_z)(m,x)dm(z)\\
=&-\delta F(m,y) + \delta^2 F(m,y,x)-\int_\r  \delta^2 F(m,z,x)dm(z).
\end{align*}

Recalling that
$$\delta^2 F(m,x,y) =\partial^2_{\eta_1\eta_2} F(\mu(\eta_1,\eta_2))_{|\eta_1=\eta_2 = 0},$$
we have shown that
\begin{equation}\label{delta2Gschwarz}
\delta^2 F(m,x,y) = -\delta F(m,y) + \delta^2 F(m,y,x)-\int_\r  \delta^2 F(m,z,x)dm(z).
\end{equation}


Then, replacing in~\eqref{delta2Gschwarz}, the term~$\delta^2 F(m,y,x)$ of the RHS, by the whole RHS of~\eqref{delta2Gschwarz} where the variables~$x$ and~$y$ are exchanged, yields
\begin{equation}\label{zero}
-\int_\r  \delta^2 F(m,z,x)dm(z) - \delta F(m,x)-\int_\r  \delta^2 F(m,z,y)dm(z) - \delta F(m,y)=0.
\end{equation}

In particular, the quantity
$$-\int_\r  \delta^2 F(m,z,x)dm(z) - \delta F(m,x)$$
does not depend on~$x$, and can only be zero thanks to~\eqref{zero}. So~\eqref{delta2Gschwarz} becomes
$$\delta^2 F(m,x,y) - \delta F(m,x) = \delta^2 F(m,y,x) - \delta F(m,y).$$
which proves the statement of the lemma.
\end{proof}

\section{Proof of Proposition~\ref{prop:mncv}}\label{append:proofstech}

Let $m\in\P_1(\r)$ whose support is included in $[-K,K]$, and $\phi:\r\rightarrow\r$ be Lipschitz continuous with Lipschitz constant non-greater than one, and let us note $N=n K$. Using the fact that $(\psi_{n,k})_{-N\leq k\leq N}$ is a partition of unity, we have
\begin{align*}
\ll|\int_\r \phi(x) d(m-m^{[n]})(x)\rr|=& \ll| \int_{\r}\phi(x)dm(x) - \sum_{k=-N}^N \left(\int_\r \psi_{n,k}(x)~dm(x)\right)\phi(k/n)\rr|\\
=&\ll|\sum_{k=-N}^N\int_{\r} \psi_{n,k}(x)\phi(x)dm(x) - \sum_{k=-N}^N \int_{\r}\phi(k/n)\psi_{n,k}(x)dm(x)\rr|.
\end{align*}

Since the support of each $\psi_{n,k}$ is included in~$I_{n,k}$ ($-N\leq k\leq N$), we obtain
\begin{align}
\ll|\int_\r \phi(x) d(m-m^{[n]})(x)\rr|\leq& \sum_{k=-N}^N \int_{I_{n,k}}\psi_{n,k}(x)\cdot\ll|\phi(x) - \phi(k/n)\rr|dm(x)\nonumber\\
\leq& \sum_{k=-N}^N \int_{I_{n,k}}\psi_{n,k}(x)\cdot\ll|x - \frac{k}{n}\rr|dm(x).\label{controlmmn}
\end{align}

On one hand, for all $k\in\llbracket -N+1,N-1\rrbracket$,
$$\int_{I_{n,k}}\psi_{n,k}(x)\cdot\ll|x - \frac{k}{n}\rr|dm(x) \leq \frac1n\int_{I_{n,k}}\psi_{n,k}(x)dm(x) = \frac1n\int_{\r}\psi_{n,k}(x)dm(x),$$
whence
\begin{equation}\label{controlmmn1}
\sum_{k=-N+1}^{N-1}\int_{I_{n,k}}\psi_{n,k}(x)\cdot\ll|x - \frac{k}{n}\rr|dm(x) \leq \frac1n \int_\r\sum_{k={-N+1}}^{N-1}\psi_{n,k}(x)dm(x)\leq \frac1n.
\end{equation}

On the other hand,
\begin{align}
\int_{I_{n,N}} \psi_{n,N}(x)\ll|x-\frac{N}{n}\rr|dm(x) \leq& \int_{](N-1)/n,+\infty[}\ll|x-\frac{N}{n}\rr|dm(x)\leq \frac1n + \int_{]N/n,+\infty[}\ll(x-\frac{N}{n}\rr)dm(x)=\frac1n.\label{controlmn2}
\end{align}
And, with the same calculation,
\begin{equation}\label{controlmn3}
\int_{I_{n,-N}} \psi_{n,-N}(x)\ll|x-\frac{-N}{n}\rr|dm(x) \leq\frac1n.
\end{equation}


Then, using inequalities~\eqref{controlmmn1},~\eqref{controlmn2} and~\eqref{controlmn3} to control~\eqref{controlmmn}:
$$\ll|\int_\r \phi(x)d(m-m^{[n]})(x)\rr|\leq \frac3n.$$

As the inequality above holds true for any Lipschitz continuous function~$\phi$ with Lipschitz constant non-greater than one, the result follows from Kantorovich-Rubinstein duality.

\end{appendix}

\bibliography{biblio}

\end{document}